\documentclass{article}[12pt]
\usepackage{amsmath,amsthm,amssymb}
\usepackage[all]{xy}
\usepackage{color}

\theoremstyle{definition}
\newtheorem{Thm}{Theorem}[section]
\newtheorem{Def}[Thm]{Definition}
\newtheorem{Lem}[Thm]{Lemma}

\newtheorem{Ex}[Thm]{Example}
\newtheorem{Prop}[Thm]{Proposition}
\newtheorem{Rem}[Thm]{Remark}

\def\Z{\mathbb{Z}}
\def\N{\mathbb{N}}

\title{The coarse Baum-Connes conjecture 
for Busemann non-positively curved spaces}
\date{}
\author{Tomohiro Fukaya\footnote
{Mathematical institute, Tohoku University, Sendai 980-8578, Japan, 
\texttt{tomo@math.tohoku.ac.jp}}, 
Shin-ichi Oguni\footnote
{Department of Mathematics, Faculty of Science, Ehime University,
Matsuyama 790-8577, Japan,
\texttt{oguni@math.sci.ehime-u.ac.jp}}}

\begin{document}

\maketitle
%%%%%%%%%%%%%%%%%%%%%%%%%%%%%%%%%%%%%%%%%%%%%%%%%
\begin{abstract}
We prove that the coarse assembly maps for
proper metric spaces which are non-positively curved  
in the sense of Busemann are isomorphisms, 
where we do not assume that the spaces are with bounded coarse geometry. 
Also it is shown that we can calculate the coarse K-homology and the K-theory of the Roe algebra
by using the visual boundaries.  \\


\noindent
2010MSC:
58J22.
\end{abstract}

\section{Introduction}\label{1}
For a proper metric space $X$, 
the coarse assembly map 
\[
\mu(X):KX_*(X)\to K_*(C^\ast X)
\]
is defined as `a coarse index map' (\cite{H-R95}, \cite{Roe93} and \cite{Yu95}).
Here $KX_*(X)$ is the coarse $K$-homology of $X$ and 
$K_*(C^\ast X)$ is the $K$-theory of the Roe algebra $C^\ast X$. 
It is well-known that $KX_*(X)$, $K_*(C^\ast X)$ and $\mu(X)$
depend only on the large-scale geometry of $X$, that is, 
the coarse structure given by the metric on $X$ (refer to \cite[Section 6]{H-R95}). 
It is interesting to know 
which proper metric space $X$ satisfies the property that 
$\mu(X)$ is an isomorphism (resp. a rational injection)
because the property is closely related to 
the analytic Novikov conjecture 
(resp. non-permission of uniformly positive scalar curvatures). 
In particular the following is called 
the coarse Baum-Connes conjecture: 
if a proper metric space $X$ is reasonable, 
for example uniformly contractible 
and with bounded coarse geometry, then 
$\mu(X)$ is an isomorphism
(See Appendix for bounded coarse geometry). 
Now we can recognize many spaces 
whose coarse assembly maps are isomorphisms.
For example, if $X$ is with bounded coarse geometry and
coarsely embeddable into a Hilbert space, 
then $\mu(X)$ is an isomorphism (\cite[Theorem 1.1]{Yu00}). 
%If $X$ has Property A, then $X$ can be embedded into a separable Hilbert space (\cite[Theorem 2.2]{Yu00}). 
%There is a class which contains spaces without bounded coarse geometry. Indeed 
Also Higson and Roe showed that 
if $X$ is geodesic and hyperbolic in the sense of Gromov, 
then $\mu(X)$ is an isomorphism (\cite[(8.2) Corollary]{H-R95}), 
where bounded coarse geometry is not assumed. 

%For some proper metric spaces, some properties depending on the large-scale geometry can be captured by some boundaries. 
Now we consider 
which proper metric space $X$ has a boundary with informations of 
$KX_*(X)$ and $K_*(C^\ast X)$.
This question is a variant of Weinberger conjecture 
which gives a sufficient condition for 
the property that the assembly map is rationally injective 
(see \cite[(6.33) Conjecture and (6.34) Proposition]{Roe93}).
Let $W$ be a boundary which gives a coarse compactification of a proper metric space $X$ 
(see Section \ref{2}). 
Then we have the following commutative diagram: 
\[
\xymatrix{
&KX_*(X) \ar[d]_{T_W} \ar[r]^{\mu(X)} 
& K_*(C^\ast X) \ar[dl]^{b_W}
\\
&\widetilde{K}_{*-1}(W) 
&.  
}
\]
Here $\widetilde{K}_{*-1}(W)$ 
is the reduced $K$-homology of $W$, 
$T_W$ is a transgression map
and $b_W$ is a map defined in \cite[Appendix]{H-R95}
(refer to \cite[Section 1.2]{F-O13}). 
There are some classes whose element $X$ 
has a coarse compactification $X\cup W$ 
satisfying that $T_W$ and $b_W$ are isomorphisms. 
Such a typical class consists of unbounded 
proper geodesic hyperbolic spaces in the sense of Gromov
(see \cite[Corollary 5.3]{F-O13}). 
Indeed the Gromov completions are desired coarse compactifications. 
See \cite{F-O13} for other classes.
%Other classes were given in \cite[Proposition A.1 and Corollary 9.1]{F-O13}. 
%Indeed lattices of Carnot groups 
%and non-uniform lattices of rank one symmetric spaces of non-compact type 
%are examples. 

In this note, we deal with proper metric spaces 
which are non-positively curved in the sense of Busemann 
(for short, proper Busemann spaces). 
Proper $CAT(0)$-spaces are 
typical examples.
The following is our main theorem: 
\begin{Thm}\label{main}
Let $X$ be an unbounded proper Busemann space
(which is not necessarily with bounded coarse geometry) 
and $\partial_v X$ be the visual boundary. 
Then $\mu(X)$ is an isomorphism.
Also $X\cup \partial_vX$ is a coarse compactification of $X$, 
and moreover $T_{\partial_vX}$ and $b_{\partial_vX}$ are isomorphisms. 
\end{Thm}

\begin{Rem}
A key in our proof of Theorem \ref{main} 
is to show that the coarsening map
\[c(X):K_*(X)\to KX_*(X)\]
is an isomorphism for any proper Busemann space $X$. 
In general, when a proper metric space 
is uniformly contractible and with bounded coarse geometry, 
the coarsening map is an isomorphism 
(see \cite[(3.8) Proposition]{H-R95}, \cite[Proof of Theorem 4.8]{E-M06} and \cite[Section 3.2]{F-O13}). 
Proper Busemann spaces are uniformly contractible, 
but they are not necessarily with bounded coarse geometry
(see Example \ref{ex'}). 
In \cite[the below of (7.5) Corollary]{H-R95}, Higson and Roe claimed that 
a proper `non-positively curved' space $X$ 
without necessarily bounded coarse geometry 
satisfies that the coarsening map is an isomorphism, 
but they did not give details. 
Indeed we prove it for proper Busemann spaces (Proposition \ref{coarsening}). 
\end{Rem}

\section{The visual boundaries of proper Busemann spaces}\label{2}
In this section, we study the visual boundaries of proper Busemann spaces
from a coarsely geometric viewpoint. 
We refer to \cite{Wil09} which deals with $CAT(0)$-spaces (see also \cite{Wil13}).

The following is one of equivalent definitions of Busemann spaces
(\cite[Proposition 8.1.2 (viii)]{Pap05}).
\begin{Def}\label{Busemann}
Let $X$ be a space endowed with a metric $d$. 
We call $X$ a {\it non-positively curved space in the sense of Busemann}, 
for short, a {\it Busemann space} 
if $X$ is geodesic and satisfies the following: 
for any points $x_0,x_1,y_0,y_1\in X$, any geodesics $\overline{x_0x_1}, \overline{y_0y_1}$ 
and any $t\in [0,1]$, 
$d(x_t,y_t)$ is smaller than or equal to $(1-t)d(x_0,y_0)+td(x_1,y_1)$, 
where $x_t\in \overline{x_0x_1}$ and $y_t\in \overline{y_0y_1}$ with 
$d(x_0,x_t)=td(x_0,x_1)$ and $d(y_0,y_t)=td(y_0,y_1)$, respectively. 
\end{Def}
\noindent 
$CAT(0)$-spaces and strictly convex Banach spaces 
like $l^p$-spaces ($1<p<\infty$) are 
typical examples of Busemann non-positively curved spaces
(see also Example \ref{ex}). 

We consider the visual boundaries. 
Let $X$ be a proper Busemann space. 
Fix a basepoint $o\in X$. 
Since two points in a Busemann space are connected by a unique geodesic, 
we have a continuous map 
\[X\times [0,1]\ni(x,t)\mapsto \delta_t(x)\in X, \]
where $\delta_t(x)$ is characterized as a point on a geodesic from $o$ to $x$ 
with $d(o,\delta_t(x))=td(o,x)$.
For each $t\in (0,\infty)$, put $B(o,t):=\{a\in X \ | \ d(o,a)\le t\}$. 
For any $s, t\in (0,\infty)$ with $s<t$, 
we define a surjection 
$\pi_{s,t}: B(o,t) \to B(o,s)$ 
as $\pi_{s,t}(a):=a$ if $d(o,a)\le s$
and $\pi_{s,t}(a):=\delta_{s/d(o,a)}(a)$
if $d(o,a)> s$. 
Also for any $t\in (0,\infty)$, we define a surjection
$\pi_t:X \to B(o,t)$ as $\pi_{t}(a):=a$ if $d(o,a)\le t$
and $\pi_{t}(a):=\delta_{t/d(o,a)}(a)$ if $d(o,a)> t$.
We consider a projective system consisting of
$\{\pi_{s,t}: B(o,t) \to B(o,s)\}_{0<s<t}$. Then the projective 
limit $\underleftarrow{\lim}B(o,t)$ contains $X$ as an open dense subset 
by the map $\underleftarrow{\lim}\pi_t:X\to \underleftarrow{\lim}B(o,t)$.
We put $\overline{X}:=\underleftarrow{\lim}B(o,t)$ and 
$\partial_v X:=\overline{X}\setminus X$.
We call $\partial_v X$ the {\it visual boundary} of $X$. 
The visual boundary of $X$ is independent (up to canonical homeomorphisms) 
of choice of basepoints by \cite[Main theorem]{Hot97}. 

\begin{Ex}\label{ex}
Let $p$ belong to $(1,\infty)$. 
For each positive integer $n$, 
denote by $\ell_p(n)$ the $n$-dimensional $\ell_p$-space. 
Then $\ell_p(n)$ is a proper Busemann space and 
the visual boundary is homeomorphic to the $(n-1)$-dimensional sphere $S^{n-1}$. 
We remark that $\ell_p(n)$ is $CAT(0)$ if and only if $n=1$ or $p=2$. 

For each $p\in (1,\infty)$, 
we give an example which is a proper Busemann space without bounded coarse geometry. 
First we consider a half-line $[0,\infty)$ with the standard metric by absolute values.   
Next we identify the zero vector of $\ell_p(n)$ with $n\in [0,\infty)$ 
for all positive integers $n$. 
Finally we endow the resulting space $X_p$ with the path-metric, 
where the embeddings $\ell_p(n)\to X_p$ for all $n\in \N$ 
and $[0,\infty)\to X_p$ are isometric. 
Then $X_p$ is a proper Busemann space 
and the visual boundary of $X_p$ is homeomorphic to 
the one-point compactification of $\bigsqcup_{n\in\N} \partial_v\ell_p(n)$, 
that is, 
the one-point compactification of $\bigsqcup_{n\in\N} S^{n-1}$. 
Obviously $X_p$ is $CAT(0)$ if and only if $p=2$.
In Appendix we show that $X_p$ is without bounded coarse geometry.
\end{Ex}

The following is implied by 
the argument in \cite[Proof of Lemma 4.6.1]{Wil09} for proper $CAT(0)$-spaces and the visual boundaries. 
\begin{Prop}\label{contractible}
Let $X$ be a proper Busemann space. 
Then $X\cup \partial_vX$ is contractible.
\end{Prop}

Now we recall a definition of coarse compactifications.
Let $X$ be a proper metric space. 
We denote by $C_b(X)$ the set of all bounded continuous functions of $X$. 
An element $f\in C_b(X)$ is a Higson function 
if for any $\epsilon>0$ and any $R> 0$, there exists a bounded set $K\subset X$
such that $x,y\in X$ with $d(x,y)<R$ and $x,y\not\in K$ 
satisfy $\left| f(x)-f(y)\right|<\epsilon$.
We denote by $C_h(X)$ the set of all Higson functions of $X$. 
Let $\overline{X}$ be a compact metrizable space which is a compactification of $X$.
We denote by $C(\overline{X})$ 
the set of all continuous functions of $\overline{X}$. 
Then $C(\overline{X})$ is naturally identified with a subalgebra of $C_b(X)$ 
by the Gelfand-Naimark theorem. 
After the identification, if every element of $C(\overline{X})$ 
is a Higson function, then we call 
$\overline{X}$ a {\it coarse compactification} of $X$. 

The following is known for proper $CAT(0)$-spaces and the visual boundaries 
(\cite[Lemma 4.6.2]{Wil09}). 
\begin{Prop}\label{visual}
Let $X$ be a proper Busemann space. 
Then $X\cup \partial_vX$ is a coarse compactification of $X$.
\end{Prop}
\begin{proof}
We need to modify the argument in \cite[Proof of Lemma 4.6.2]{Wil09} 
because his proof uses the $CAT(0)$-inequality.
%In fact the following proof is parallel 
%to the argument in \cite[Proof of Proposition A.1]{F-O13}. 

We recall that $X\cup \partial_vX=\underleftarrow{\lim}B(o,t)$. 
We naturally regard $C(B(o,t))$ as a subset of $C_b(X)$ by 
\[C(B(o,t))\ni f\mapsto f\circ \pi_t\in C_b(X), \]
where $C(B(o,t))$ is the set of all continuous functions on $B(o,t)$. 
Then $C(X\cup \partial_vX)$ is identified with the closure of $\bigcup_{t>0} C(B(o,t))$.
We take $t\in (0,\infty)$ and $f\in C(B(o,t))$.  
Then it is enough to prove that 
$F:=f\circ \pi_t$ is a Higson function of $X$. 
We take $\epsilon>0$ and $R>0$. 
Since $B(o,t+R)$ is compact, we have $\delta>0$
such that $|F(a')-F(b')|<\epsilon$
for any $a',b'\in B(o,t+R)$ with $d(a',b')<\delta$.
We take any $a,b\in X$ with $d(a,b)<R$ and have
\[
r:=\min\{r_a:=d(o,a),r_b:=d(o,b)\}>S:=\max\{t, tR/\delta\}.
\] 
We put $a'=\delta_{t/r} (a)$ and $b'=\delta_{t/r} (b)$.
Then we have 
\[t\le d(o,a')=tr_a/r< t+R,\, t\le d(o,b')=tr_b/r< t+R\]
and $F(a')=F(a),\, F(b')=F(b)$.
Since $X$ is a Busemann space, we have 
\[
d(a',b')=d(\delta_{\frac{t}{r}} (a),\delta_{\frac{t}{r}} (b))
\le \frac{t}{r}d(a,b)<\frac{tR}{r}\le \delta.  
\]
Hence we have $|F(a)-F(b)|< \epsilon$. 
\end{proof}

\section{The coarse assembly maps for proper Busemann spaces}\label{3}
In this section we prove that 
the coarse assembly maps for proper Busemann spaces are isomorphisms. 

Let $X$ be a proper metric space. 
Then we have the following commutative diagram: 
\begin{equation}\label{comm}
\xymatrix{
&K_*(X) \ar[dr]^{c(X)} \ar[rr]^{A(X)} & \ar[r] 
& K_*(C^\ast X) \\
&&KX_*(X) \ar[ur]^{\mu(X)}  
&. 
}
\end{equation}
Here $K_*(X)$ is the $K$-homology of $X$, 
$C^\ast X$ is the Roe algebra of $X$, 
$K_*(C^\ast X)$ is the $K$-theory of $C^\ast X$, 
$A(X)$ is the assembly map for $X$, 
$c(X)$ is the coarsening map for $X$ and 
$\mu(X)$ is the coarse assembly map for $X$ 
(see \cite[Section 6]{H-R95} and also \cite[Section 1]{F-O13}). 

If a proper metric space $X$ is a scalable space 
(which is not necessarily with bounded coarse geometry), 
then Higson and Roe proved that the assembly map $A(X)$ is an isomorphism
by a so-called Eilenberg swindle (see \cite[(7.2) Theorem]{H-R95} and \cite[12.4.11 Theorem]{H-R00}). 
Proper $CAT(0)$-spaces 
are such examples (see \cite[Exercise 12.7.4]{H-R00} and \cite[Proof of Proposition 4.6.3]{Wil09}).
Their proof works well for proper Busemann spaces and thus 
we have the following: 
\begin{Prop}\label{scalable}
Let $X$ be a proper Busemann space. 
Then $X$ is scalable and thus $A(X)$ is an isomorphism. 
\end{Prop}

We recall definitions of coarse maps. 
Let $(X,d_X)$ and $(Y,d_Y)$ be proper metric spaces.
A map $f: Y\to X$ is a {\it coarse map} if $f$ satisfies the following:
\begin{itemize}
\item for any bounded subset $K$ of $X$, $f^{-1}(K)$ is bounded; 
\item for any $R>0$, there exists $S>0$ such that $y,y'\in Y$ with $d_Y(y,y')<R$
satisfy $d_X(f(y),f(y'))<S$. 
\end{itemize}
Two coarse maps $f,g:Y\to X$ are {\it close} if there exists $C>0$ such that 
$d_X(f(y),g(y))<C$ for any $y\in Y$. 

We need two lemmas in order to prove Proposition \ref{coarsening}.
\begin{Lem}\label{key}
Let $X$ be a proper Busemann space, 
$Y$ be a connected locally finite simplicial complex with the spherical metric
and $f:Y\to X$ be a coarse map. 
Then there exists a continuous coarse map $g:Y\to X$ 
such that $g$ and $f$ are close. 
\end{Lem}
\noindent
Here we briefly explain what is the spherical metric.
Let $Y$ be a connected locally finite simplicial complex. 
Denote the $0$-th skeleton by $Y^{(0)}:=\{v_i\}_{i\in \N}$. 
Then for every non-negative integer $k$, 
we identify every $k$-simplex with vertices $\{v_{i_1},\ldots, v_{i_{k+1}}\}$ of $Y$ with 
\[
\Delta_S^k:=\left\{(u_{i_1},\ldots, u_{i_{k+1}})\in [0,1]^{k+1} \left| \sum_{j=1}^{k+1}u_{i_j}^2=1 \right.\right\}.
\]
By using the identification, 
we have so called spherical coordinates   
$\{u_i:Y\to [0,1]\}_{i\in \N}$, where $\sum_{i\in \N}u_i(y)^2$ is a finite sum and equal to $1$. 
Now we have an embedding 
$Y\ni y\mapsto (u_1(y), u_2(y), \ldots)\in S(\ell_2(\N))$,
where $S(\ell_2(\N))$ is the sphere with radius $1$ in $\ell_2(\N)$. 
We endow $Y$ with the metric $d_S$ 
induced from the path-metric of the image of the embedding 
and call $d_S$ the {\it spherical metric}. 
The spherical metric is very familiar with the graph metric in the following sense: 
the restriction of $d_S$ on the first skeleton $Y^{(1)}$
is equal to the graph metric on $Y^{(1)}$ up to $\pi/2$-multiplication.

\begin{proof}[Proof of Lemma \ref{key}]
We denote by $Y^{(0)}$ the $0$-th skeleton of $Y$. 
Then we construct a continuous coarse map $g:Y\to X$ 
with $g|_{Y^{(0)}}=f|_{Y^{(0)}}$ 
such that $g$ and $f$ are close. 

Every point $y$ of $Y$ is uniquely presented 
by the barycentric coordinate $\{(v_1,t_1),\cdots,(v_n,t_n)\}$, 
where $v_i\in Y^{(0)}$ and $t_i\in (0,1]$ with $\sum_{i=1}^n t_i=1$. 
Now for such a point $y$, we define $g(y)\in X$ as 
the barycenter of $\{(f(v_1),t_1),\cdots,(f(v_n),t_n)\}$, 
that is, $g(y)$ is a unique point which attains the minimum of 
\[X\ni x\mapsto \sum_{i=1}^nt_id_X(x,f(v_i))^2\in [0,\infty)\] 
(see \cite[Definition 8.4.10]{Pap05}). 
Then we have a map $g:Y\to X$ with $g|_{Y^{(0)}}=f|_{Y^{(0)}}$.
Since the barycenter continuously depends on weights $t_i$, 
the map $g$ is continuous. 

Now we prove that $g$ is close to $f$. 
Since $f$ is coarse, we have $C>0$ such that 
for any point $y,y'\in Y$ contained in a same simplex, 
$d_X(f(y),f(y'))<C$. 
Take a point $y\in Y$ which is presented 
by the barycentric coordinate $\{(v_1,t_1),\cdots,(v_n,t_n)\}$. 
Then we have $i_0\in\{1,\ldots,n\}$ satisfying $d(g(y),f(v_{i_0}))\le d(g(y),f(v_i))$ 
for any $i\in\{1,\ldots,n\}$. 
Since $g(y)$ is the barycenter of 
$\{(f(v_1),t_1),\cdots,(f(v_n),t_n)\}$ with $\sum_{i=1}^nt_i=1$, 
we have 
\begin{align*}
d(g(y),f(v_{i_0}))^2
=&\sum_{i=1}^n t_id(g(y),f(v_{i_0}))^2
\le \sum_{i=1}^n t_id(g(y),f(v_i))^2 \\
< &\sum_{i=1}^n t_i d(f(v_{i_0}),f(v_i))^2
\le \sum_{i=1}^n t_i  C^2
=C^2.
\end{align*}
Since we have $d(f(y), f(v_{i_0}))<C$ and $f(v_{i_0})=g(v_{i_0})$, 
the triangle inequality implies $d(g(y), f(y))<2C$. 

Since $g$ is close to the coarse map $f$, $g$ is also a coarse map. 
\end{proof}

\begin{Lem}\label{homotopy}
Let $X$ be a proper Busemann space and 
$\phi :X\to X$ be a continuous coarse map which is close to $id_X$. 
Then $id_X$ and $\phi$ are properly homotopic. 
\end{Lem}
\begin{proof}
We consider the map 
\[
h:X\times [0,1]\ni (x,t)\mapsto h(x,t)\in X, 
\]
where $h(x,t)$ is characterized as a point on the geodesic from $x$ to $\phi(x)$
with $d(x,h(x,t))=td(x,\phi(x))$. 
Then $h$ is continuous in view of Definition \ref{Busemann}. 
%Also $h$ is proper and satisfies $h(-,0)=id_X$ and $h(-,1)=\phi$. 
Hence $h$ is a proper homotopy between $id_X$ and $\phi$. 
\end{proof}

\begin{Prop}\label{coarsening}
Let $X$ be a proper Busemann space. 
Then $c(X)$ is an isomorphism. 
\end{Prop}
\begin{proof}
The argument in \cite[Proof of (3.8) Proposition]{H-R95} 
works well on our situation
when we use Lemmas \ref{key} and \ref{homotopy} 
instead of \cite[(3.3) Lemma and (3.4) Lemma]{H-R95}.   
We give a sketch for reader's convenience. 
First of all, we take an anti-\v{C}ech system $\mathcal{U}_1,\mathcal{U}_2,\ldots$ 
of $X$ and a partition of unity of $\mathcal{U}_1$. 
When we denote by $|\mathcal{U}_i|$ the geometrization of $\mathcal{U}_i$ as a nerve complex
for each $i\in \N$, 
we have a coarsening sequence: 
\[
\xymatrix{
&X \ar[r]^{\varphi}& |\mathcal{U}_1|\ar[r]^{\varphi_1}& |\mathcal{U}_2|\ar[r]^{\varphi_2}& \cdots ,  
}
\]
where all maps are proper continuous maps and coarse equivalences. 
By Lemma \ref{key}, we have a continuous coarse map 
$g_i:|\mathcal{U}_i|\to X$ for each $i\in \N$ satisfying that 
$g_i\circ (\varphi_{i-1}\circ\cdots\circ\varphi_1\circ \varphi)$ and 
$(\varphi_{i-1}\circ\cdots\circ\varphi_1\circ \varphi)\circ g_i$ 
are close to $\mathrm{id}_X$ and $\mathrm{id}_{|\mathcal{U}_i|}$, respectively.
Then Lemma \ref{homotopy} implies that 
$g_i\circ (\varphi_{i-1}\circ\cdots\varphi_1\circ \varphi)$ and $\mathrm{id}_X$ are properly homotopic. 
Since $(\varphi_{i-1}\circ\cdots\circ\varphi_1\circ \varphi)\circ g_i$ and $\mathrm{id}_{|\mathcal{U}_i|}$ are close,
we have $k\ge i$ such that $(\varphi_{k}\circ\cdots\circ\varphi_1\circ \varphi)\circ g_i$
and $\varphi_k\circ\cdots\circ\varphi_i$ are contiguous and thus properly homotopic.
Then the coarsening sequence implies that 
$K_*(X)\cong \underrightarrow{\lim}{K_*(|\mathcal{U}_i|)}$, where the right hand side is $KX_*(X)$ 
by definition.  
\end{proof}

Propositions \ref{scalable} and \ref{coarsening} 
imply that $\mu(X)$ is an isomorphism for 
a proper Busemann space (see the commutative diagram (\ref{comm})).

\section{Proof of Theorem \ref{main}}\label{4}
In this section we complete a proof of Theorem \ref{main}. 

When $X$ is an unbounded proper metric space and 
$X\cup W$ is a coarse compactification of $X$, 
we have the following commutative diagram: 
\[
\xymatrix{
&K_*(X) \ar[dr]^{c(X)} \ar[rr]^{A(X)} \ar[ddr]_{\partial_W}
& \ar[r] 
& K_*(C^\ast X) \ar[ddl]^{b_W}
\\
&
&KX_*(X) \ar[ur]^{\mu(X)} \ar[d]^{T_W} 
& 
\\
&
&\widetilde{K}_{*-1}(W)
&, 
}
\]
where $\partial_W$ is a boundary map of the homological long exact sequence
for a pair $(X\cup W, W)$ and $b_W$ is a map defined in \cite[Appendix]{H-R95}
(see also \cite[Section 1.2]{F-O13}).  

\begin{proof}[Proof of Theorem \ref{main}]
Let $X$ be an unbounded proper Busemann space and consider the visual boundary $\partial_v X$. 
We already proved that $A(x)$, $c(X)$ and $\mu(X)$ are isomorphisms in Section \ref{3}. 
Since $X\cup \partial_vX$ is contractible by Proposition \ref{contractible},
$\partial_{\partial_vX}$ is an isomorphism. 
Now the assertion follows from the commutative diagram in the above. 
\end{proof}

\begin{Ex}\label{ex''}
Let $p$ belong to $(1,\infty)$. 
We consider $X_p$ in Example \ref{ex}.
Then Theorem \ref{main} implies the following: 
\begin{align*} 
&KX_*(X_p)\cong K_*(C^\ast X_p)\cong \widetilde{K}_{*-1}(\partial_vX_p)\\
\cong &K_{*-1}(\bigsqcup_{n\in\N} S^{n-1})\cong \prod_{n\in \N}K_{*-1}(S^{n-1})\cong \prod_{n\in\N} \Z,  
\end{align*}
where we used the strong excision property and the cluster axiom 
(see \cite[Definition 7.3.1]{H-R00}).
\end{Ex}

\section{Coarse $K$-theories and coarse cohomologies of proper Busemann spaces}\label{5}
In this section we briefly consider coarse $K$-theories and coarse cohomologies of proper Busemann spaces. 

Let $X$ be a proper metric space and $X\cup W$ be a coarse compactification of $X$. 
Then we have the following commutative diagram: 
\[
\xymatrix{
&K^*(X)  
&  \ar[l]
& K_{*-1}(\mathfrak c^r X) \ar[ll]_{A(X)} \ar[dl]_{\mu(X)}
\\
&
&KX^*(X) \ar[ul]_{c(X)}  
& 
\\
&
&\widetilde{K}^{*-1}(W) 
\ar[uul]^{\partial_{W}} \ar[u]_{T_{W}} \ar[uur]_{b_{W}}. 
& 
}
\]
Here 
$K^*(X)$, 
$\mathfrak c^r X$, 
$K_*(\mathfrak c^r X)$, 
$KX^*(X)$,
$\widetilde{K}^*(W)$, 
$A(X)$,
$\mu(X)$,
$c(X)$,
$\partial_W$, 
$b_W$ 
and $T_W$ are 
the $K$-theory of $X$, 
the reduced stable Higson corona of $X$, 
the $K$-theory of $\mathfrak c^r X$, 
the coarse $K$-theory of $X$, 
the reduced $K$-theory of $W$, 
the co-assembly map of $X$, 
the coarse co-assembly map of $X$, 
the character map of $X$, 
the boundary map of the cohomological long exact sequence of $(X\cup W, W)$, 
the map induced by the inclusion of $C(W)$ into $\mathfrak c^r X$ 
and the transgression map, respectively
(see \cite[Section 4]{E-M06}, \cite[Chapter 4]{Wil09}, \cite[Sections 3, 4]{F-O13} for details). 

When we consider a proper Busemann space $X$ and the visual boundary $\partial_v X$, 
we have the following by a similar argument in Proof of Theorem \ref{main}, but we omit details. 
\begin{Thm}
Let $X$ be an unbounded proper Busemann space
(which is not necessarily with bounded coarse geometry) 
and $\partial_v X$ be the visual boundary.  
Then in the above commutative diagram, 
all maps $A(X)$, $c(X)$, $\mu(X)$, $\partial_{\partial_v X}$, $b_{\partial_v X}$ and $T_{\partial_v X}$
are isomorphisms. 
\end{Thm}
\noindent
If $X$ is an unbounded proper $CAT(0)$-space with bounded coarse geometry, then 
the above was known (see \cite[Theorem 4.8]{E-M06} and \cite[Section 4.6]{Wil09}).

Similarly we have the following: 
\begin{Thm}
Let $X$ be an unbounded proper Busemann space
(which is not necessarily with bounded coarse geometry) 
and $\partial_v X$ be the visual boundary. 
Then all maps in the commutative diagram 
\[
\xymatrix{
&H_c^*(X) &HX^*(X) \ar[l]_{c(X)}\\  
&       &\widetilde{H}^{*-1}(\partial_vX) 
\ar[ul]^{\partial_{\partial_vX}} \ar[u]_{T_{\partial_vX}}  
}
\]
are isomorphisms. 
Here $H_c^*(X)$,  
$HX^*(X)$,
$\widetilde{H}^*(W)$, 
$c(X)$,
$\partial_W$, 
and $T_W$ are 
the compactly supported Alexander-Spanier cohomology of $X$, 
the coarse cohomology of $X$, 
the reduced Alexander-Spanier cohomology of $W$, 
the character map of $X$, 
the boundary map of the cohomological long exact sequence of $(X\cup W, W)$,  
and the transgression map, respectively
(see \cite{Roe93} and \cite[Section 3]{F-O13} for details). 
\end{Thm}
\noindent
In the above, we can know that $c(X)$ is an isomorphism by another reason.
Indeed it is known that the compactly supported Alexander-Spanier cohomology and 
the coarse cohomology are isomorphic by the character map
for any uniformly contractible proper metric space without necessarily bounded coarse geometry
(\cite[(3.33) Proposition]{Roe93}). 
We note that Busemann spaces are uniformly contractible.

\appendix
\section{Bounded coarse geometry}

In this appendix we collect some equivalent definitions of bounded coarse geometry
for reader's convenience (see also \cite[Section 3.1]{Roe03})
and prove that $X_p$ in Example \ref{ex} is without bounded coarse geometry.

\begin{Def}\label{bounded}
Let $\Gamma$ be a discrete metric space. 
Then $\Gamma$ is said to be {\it with bounded geometry} if 
for any $R>0$,  
\[
\sup\{\# B(\gamma,R) \left| \gamma\in \Gamma \right.\}<\infty.
\]
\end{Def}
\noindent 
When a discrete metric space is with bounded geometry, 
so is every subset with the restricted metric.  

The following was introduced in Fan's PhD thesis 
(see \cite[(3.6) Definition]{H-R95} and also \cite[Definition 3.9]{Roe03}). 
\begin{Def}\label{bounded coarse}
Let $X$ be a metric space. 
Then $X$ is said to be {\it with bounded coarse geometry} if 
there exists $\epsilon>0$ satisfying the following: 
for any $R>0$,
\[
\sup\{l \left| x\in X, x_1,\ldots,x_l \in B(x,R), i\neq j, d(x_i,x_j)>\epsilon \right.\}<\infty.
\] 
\end{Def}
\noindent
When a metric space is with bounded coarse geometry, 
so is every subset with the restricted metric. 

Bounded coarse geometry is simply called bounded geometry 
in many literatures.  
However a discrete metric space with bounded coarse geometry 
is not necessarily with bounded geometry 
in the sense of Definition \ref{bounded}.
Therefore we adopt the notion of bounded coarse geometry. 

\begin{Def}
Let $X$ be a metric space and $\Gamma$ be a discrete subset of $X$. 
For $\epsilon> 0$ and $C\ge 0$, $\Gamma$ is said to be 
{\it $\epsilon$-separated} if 
every pair of two distinct elements 
$\gamma_1, \gamma_2\in \Gamma$ satisfies $d(\gamma_1,\gamma_2)> \epsilon$.
For $C\ge 0$, $\Gamma$ is said to be 
{\it $C$-dense} if 
every $x\in X$ satisfies $d(x,\gamma)\le C$ for some $\gamma\in \Gamma$.
Also $\Gamma$ is called a {\it net} of $X$ 
if $\Gamma$ is $C$-dense for some $C\ge 0$. 
\end{Def}
\noindent 
We can show that every metric space has an $\epsilon$-separated net for any $\epsilon> 0$
by using Zorn's lemma. 

\begin{Prop}
Let $X$ be a metric space. 
The the following are equivalent.
\begin{enumerate}
\item[(i)] $X$ is with bounded coarse geometry.  
\item[(ii)] $X$ has a net with bounded geometry. 
\item[(iii)] $X$ is coarsely equivalent to 
a discrete metric space with bounded geometry. 
\item[(iv)] there exists $\epsilon> 0$ such that 
every $\epsilon'$-net of $X$ for any $\epsilon' \ge \epsilon$
is with bounded geometry,
\item[(v)] there exists $\epsilon>0$ satisfying the following: 
for any $R>0$, there exists $N>0$ such that 
for any $x\in X$, 
\[
B(x,R)\subset \bigcup_{i=1,\ldots,N}B(x_i,\epsilon)
\]
for some $x_1,\ldots,x_N\in X$.
\end{enumerate}
\end{Prop}
\begin{proof}
Conditions (i) and (v) are equivalent by \cite[Proposition 3.2 (d)]{Roe03}.
It is easy to show equivalence between conditions (ii) and (iii). 
Trivially condition (iv) implies condition (ii). 

Now we prove that condition (i) implies condition (iv). 
Take $\epsilon$ in Definition \ref{bounded coarse}. 
We consider any $\epsilon'$-separated net $\Gamma$ where $\epsilon'\ge \epsilon$. 
Then for any $R>0$, any $\gamma\in \Gamma$
\begin{align*}
&\#(B(\gamma,R)\cap \Gamma)\\
\le&
\sup\{l \left| x\in X, x_1,\ldots,x_l \in B(x,R), i\neq j, d(x_i,x_j)>\epsilon \right.\}.
\end{align*} 
Hence $\Gamma$ is with bounded geometry. 

Finally we show that condition (ii) implies condition (v). 
Suppose that $\Gamma$ is a $C$-dense net of $X$ for some $C\ge 0$
and with bounded geometry. 
For each $x\in X$, we chose $\gamma_x\in \Gamma$ with $d(x,\gamma_x)\le C$. 
Then for any $R>0$ and any $x\in X$, we have  
\[
B(x,R)\subset \bigcup_{\gamma \in B(x,R+C)\cap \Gamma}B(\gamma, C)
\subset \bigcup_{\gamma \in B(\gamma_x,R+2C)\cap \Gamma}B(\gamma, C). 
\]
Since $\Gamma$ is with bounded geometry, 
\[
\sup\{\#(B(\gamma_x,R+2C)\cap \Gamma) \left| x\in X\right.\}<\infty.
\]
\end{proof}

It follows from condition (iii) that 
bounded coarse geometry is a coarse geometric property, 
that is, 
for two coarsely equivalent metric spaces $X$ and $Y$, 
$X$ is with bounded coarse geometry 
if and only if 
so is $Y$. 

Conditions (ii) and (iii) are useful when we show that
some spaces are with bounded coarse geometry. 
On the other hand we can use condition (iv) to prove that 
some space is without bounded coarse geometry. 
\begin{Ex}\label{ex'}
Let $X_p$ be a metric space in Example \ref{ex}. 
Let $k$ be a positive integer. 
We take the subset $(k\Z)^n$ of $\ell_p(n)$ for $n\in\N$.  
When we naturally regard $\Gamma_k:=\sqcup_{n\in \N}(k\Z)^n$ as a subset of $X_p$, 
$\Gamma_k$ is a $(k-0.01)$-separated net of $X_p$. 
Since $\Gamma_k$ is without bounded geometry, 
it follows from condition (iv) that $X_p$ is without bounded coarse geometry.  
\end{Ex}

\section*{Acknowledgements}
The authors would like to thank Professor Hiroyasu Izeki 
for very helpful discussions. 
T. Fukaya and S. Oguni are 
supported by Grant-in-Aid for Scientific Researches for Young Scientists (B) 
(No. 23740049) and (No. 24740045) Japan Society of Promotion of Science, respectively.

%%%%%%%%%%%%%%%%%%%%%%%%%%%%%%%%%%


\begin{thebibliography}{100}


\bibitem{E-M06}
Heath Emerson, Ralf Meyer, 
Dualizing the coarse assembly map. 
J. Inst. Math. Jussieu 5 (2006), no. 2, 161--186.

%\bibitem{F-O12}
%Tomohiro Fukaya, Shin-ichi Oguni,  
%The coarse Baum-Connes conjecture for relatively hyperbolic groups. 
%J. Topol. Anal. 4 (2012), no. 1, 99--113.

\bibitem{F-O13}
Tomohiro Fukaya, Shin-ichi Oguni, 
Coronae of relatively hyperbolic groups and coarse cohomologies.
preprint, 2013, arXiv:1303.1865.

\bibitem{H-R95}
Nigel Higson, John Roe, 
On the coarse Baum-Connes conjecture. 
Novikov conjectures, index theorems and rigidity, Vol. 2 (Oberwolfach, 1993), 227--254, 
London Math. Soc. Lecture Note Ser., 227, Cambridge Univ. Press, Cambridge, 1995.

\bibitem{H-R00}
Nigel Higson, John Roe, 
Analytic K -homology. 
Oxford Mathematical Monographs. Oxford Science Publications. 
Oxford University Press, Oxford, 2000. xviii+405 pp.

\bibitem{Hot97}
Philip K. Hotchkiss, 
The boundary of a Busemann space. 
Proc. Amer. Math. Soc. 125 (1997), no. 7, 1903--1912.

\bibitem{Pap05}
Athanase Papadopoulos, 
Metric spaces, convexity and nonpositive curvature. 
IRMA Lectures in Mathematics and Theoretical Physics, 6. 
European Mathematical Society (EMS), Zurich, 2005. xii+287 pp.

\bibitem{Roe93}
John Roe,  
Coarse cohomology and index theory on complete Riemannian manifolds. 
Mem. Amer. Math. Soc. 104 (1993), no. 497, x+90 pp.

\bibitem{Roe03}
John Roe, 
Lectures on coarse geometry. 
University Lecture Series, 31. American Mathematical Society, Providence, RI, 2003. viii+175 pp.

\bibitem{Wil09}
Rufus Willett, 
Band-dominated operators and the stable Higson corona. 
PhD-thesis, Penn State, 2009.

\bibitem{Wil13}
Rufus Willett,  
Some `homological' properties of the stable Higson corona. 
J. Noncommut. Geom. 7 (2013), no. 1, 203--220.

\bibitem{Yu95}
Guoliang Yu,
Coarse Baum-Connes conjecture. 
K-Theory 9 (1995), no. 3, 199--221.

\bibitem{Yu00}
Guoliang Yu, 
The coarse Baum-Connes conjecture for spaces which admit a uniform embedding into Hilbert space. 
Invent. Math. 139 (2000), no. 1, 201--240.

\end{thebibliography}
\end{document}